\documentclass[10pt,reqno]{amsart}   
 
\usepackage{amssymb,latexsym,amsmath}
\usepackage[left=1in,top=1in,right=1in,bottom=1in]{geometry} 
\usepackage{fancyhdr, lastpage} 
\usepackage{leftidx}
\usepackage{color}
 \usepackage{mathrsfs}
\usepackage{verbatim}
\usepackage{amsthm}
\usepackage{wasysym}
\usepackage{upgreek}
\usepackage{accents}
\usepackage{fancyhdr}
\usepackage{times}

\begin{document}

\sloppy
\renewcommand{\theequation}{\arabic{section}.\arabic{equation}}
\thinmuskip = 0.5\thinmuskip
\medmuskip = 0.5\medmuskip
\thickmuskip = 0.5\thickmuskip
\arraycolsep = 0.3\arraycolsep

\newtheorem{theorem}{Theorem}[section]
\newtheorem{lemma}[theorem]{Lemma}
\newtheorem{corollary}[theorem]{Corollary}
\newtheorem{proposition}[theorem]{Proposition}
\newtheorem{definition}[theorem]{Definition}
\newtheorem{remark}[theorem]{Remark}
\renewcommand{\thetheorem}{\arabic{section}.\arabic{theorem}}
\def\prf{\begin{proof}}
\def\prfe{\end{proof}}

\def\be{\begin{equation}}
\def\ee{\end{equation}}
\def\bea{\begin{eqnarray}}
\def\eea{\end{eqnarray}}
\def\beas{\begin{eqnarray*}}
\def\eeas{\end{eqnarray*}}

\newcommand{\R}{\mathbb R}
\newcommand{\N}{\mathbb N}
\newcommand{\T}{\mathbb T}
\newcommand{\K}{\mathbb S}
\newcommand{\leftexp}[2]{{\vphantom{#2}}^{#1}{#2}}

\newcommand{\todo}[1]{\vspace{5 mm}\par \noindent
\marginpar{\textsc{ \hspace{.2 in}   \textcolor{red}{ To Fix}}} \framebox{\begin{minipage}[c]{0.95
\textwidth} \tt #1 \end{minipage}}\vspace{5 mm}\par}
\newcommand{\eqdef}{\overset{\mbox{\tiny{def}}}{=}}

\def\g{\partial}
\def\E{\mathcal{ E}}
\def\D{\mathcal{D}}
\def\t{\bar{\partial}} 
\def\div{\mbox{div}}
\def\curl{\mbox{curl}}
\def\open#1{\setbox0=\hbox{$#1$}
\baselineskip = 0pt
\vbox{\hbox{\hspace*{0.4 \wd0}\tiny $\circ$}\hbox{$#1$}}
\baselineskip = 11pt\!}
\def\A{\leftexp{\kappa}{\!\!\!\!\!\!A}}
\def\a{\leftexp{\kappa}{\!\!\!a}}
\def\kA{\leftexp{\kappa}{\!\!\!\tilde{A}}}
\def\w{\leftexp{\kappa}{\!\!\!w}}
\def\kw{\leftexp{\kappa}{\!\!\!\tilde{w}}}
\def\qn{\open{q}}  
\def\bp{{\bar \partial}}
\def\lessim{\lesssim}

\title{Global stability of steady states in the classical Stefan problem}
\author{Mahir Had\v{z}i\'{c} and Steve Shkoller}
\address{Department of Mathematics,
King's College London, UK}
\email{mahir.hadzic@kcl.ac.uk}
\address{Department of Mathematics, University of California, Davis, CA 95616, USA}
\email{shkoller@math.ucdavis.edu}


\begin{abstract}
The classical one-phase Stefan problem (without surface tension) allows for a continuum of steady state solutions, given by an arbitrary (but sufficiently smooth) domain together with zero temperature.
We prove global-in-time stability of such steady states, assuming a sufficient degree of smoothness on the initial domain, but without any a priori restriction on the convexity properties of the initial shape.
This is an extension of our previous result \cite{HaSh2} in which we studied nearly spherical shapes.
\end{abstract}

\maketitle

\setcounter{tocdepth}{1}


\section{Introduction}

\subsection{The problem formulation}
We consider the problem of global existence and  asymptotic stability of classical solutions to the {\em classical} Stefan problem, which
models the  evolution of  the time-dependent phase boundary  between  liquid and solid phases.
The temperature $p(t,x)$  of the liquid 
and the {\it a priori unknown}  moving phase boundary $\Gamma(t)$ must satisfy the
following system of equations:
\begin{subequations}
\label{E:stefan}
\begin{alignat}{2}
p_t-\Delta p&=0&&\ \text{ in } \ \Omega(t)\,;\label{E:heat}\\
\mathcal{V} (\Gamma(t))&=-\g_np  && \ \text{ on } \ \Gamma(t)\,;\label{E:neumann}\\
p&=0&& \ \text{ on } \ \Gamma(t)\,;\label{E:dirichlet}\\
p(0,\cdot)=p_0\,, & \ \Omega (0)=\Omega &&\,.\label{E:initial}
\end{alignat}
\end{subequations}
For each instant of time $t \in [0,T]$, $\Omega(t)$ is a time-dependent open subset of $\R^d$ with $d\ge 2$, and $\Gamma(t)\eqdef \partial \Omega(t)$ denotes
the moving, time-dependent  free-boundary.

The heat equation (\ref{E:heat})  models thermal diffusion
 in the bulk $\Omega(t)$ with thermal diffusivity set to $1$. The boundary transport equation  (\ref{E:neumann})
states that each point on the moving boundary is transported with normal velocity equal to
 $-\g_np=-\nabla p\cdot n$,  the normal derivative of $p$ on $\Gamma(t)$.  Here, 
 $n(t, \cdot )$ denotes the outward pointing unit normal to $\Gamma(t)$,  and 
$ \mathcal{V} (\Gamma(t))$ denotes the speed or the normal velocity of the hypersurface $\Gamma(t)$.
The homogeneous Dirichlet boundary condition~(\ref{E:dirichlet}) is termed the {\em classical Stefan condition}
and problem~(\ref{E:stefan}) is called the {\em classical Stefan problem}. It implies that the
freezing of the liquid occurs at a constant temperature $p=0$. 
Finally, in~\eqref{E:initial} we specify the initial temperature distribution
$p_0:\Omega\to\R$, as well as the initial geometry  $ \Omega$.   
Because the liquid phase $\Omega(t)$ is characterized by
the set $\{ x \in \mathbb{R}  ^d \ : \ p(x,t) >0 \}$, we shall consider initial data $p_0 >0$ in $\Omega$.
 Thanks to
 (\ref{E:heat}), the parabolic Hopf lemma implies that 
$\g_np(t)<0$ on $\Gamma(t)$ for $t>0$,  so we impose the  {\it non-degeneracy condition} (also known as the
{\em Rayleigh-Taylor sign condition} in fluid mechanics~\cite{Ray,Tay,Wu97,CoSh07, CoSh12, CoHoSh12}):
\be\label{E:tay}
-\g_np_0\ge \lambda > 0 \quad\text{on}\,\,\,\,\Gamma(0)
\ee
on our initial temperature distribution.  
Under the above assumptions, we proved in~\cite{HaSh} that (\ref{E:stefan}) is locally well-posed.
%

Steady states $(\bar{u},\bar{\Gamma})$ of~\eqref{E:stefan} consist of arbitrary  domains with $\bar{\Gamma}\in C^1$ and with temperature $\bar{u}\equiv0.$
The main goal of this paper is to prove global-in-time stability of such steady states, independent of any convexity assumptions. 
Our analysis employs high-order energy spaces, which are weighted by 
the normal derivative of the temperature along the moving boundary;  we create a {\it hybridized 
energy method},  combining  integrated quantities with  {\em pointwise} methods via the Pucci extremal operators,  which allow us to track the time-decay properties of the normal derivative of the temperature. 
This hybrid approach appears to be new, and is a natural extension of our previous work~\cite{HaSh2}, which necessitated perturbations of spherical initial domains. 

\subsection{Notation}
For any $s\geq0$ and given functions $f:\Omega\to\R$, $\varphi:\Gamma\to\R$ we set
$$
\|f\|_s\eqdef \|f\|_{H^s(\Omega)} \text{ and }
|\varphi|_s\eqdef \|\varphi\|_{H^s(\Gamma)}.
$$
If $i=1,...,d$ then 
$f,_i\eqdef \g_{x^i}f$ is the partial derivative of $f$ with respect to $x^i$. Similarly,
$f,_{ij}\eqdef \g_{x^i}\g_{x^j}f$, etc.   For time-differentiation,
$f_t\eqdef \g_tf$.   Furthermore, for a function $f(t,x)$,  we shall often write $f(t)$ for $f(t, \cdot )$, and
$f(0)$ to mean $f(0,x)$.
The space of continuous functions on $\Omega$ is denoted by $C^0(\Omega).$
For any given multi-index $\alpha=(\alpha_1,\dots,\alpha_d)$ we set 
\[
\g^{\vec{\alpha}} = \g_1^{\alpha_1}\dots\g_d^{\alpha_d}.
\]
We also define the tangential gradient $\t$ by   $\t f\eqdef   \nabla f - \g_N fN,$  where $N$ stands for the outward-pointing unit normal onto $\partial\Omega$
and $\g_N f = N\cdot \nabla f$ is the normal derivative of $f.$ By extending $N$ smoothly into a neighborhood of $\Gamma$ inside the interior of $\Omega$ we can define 
$\t$ on that neighborhood  in the same way.
We employ the following notational convention:
\[
\t f= (\t_1f,\dots,\t_d f),\quad\t^{\vec{\alpha}}f\eqdef (\t_1^{\alpha_1}f,\dots,\t_d^{\alpha_d}f) \,,
\]
where $ \vec{\alpha} = (\alpha_1,\dots,\alpha_d)$ denotes a multi-index.
The identity map on $\Omega$ is denoted by $e(x)=x$, while the identity matrix is denoted by $\text{Id}$.
 We use $C$ to denote a universal (or generic) constant that may change from inequality to inequality.
We write $X\lesssim Y$ to denote 
$X\leq C\, Y$. 
We use the notation $P(s)$ to denote a generic non-zero real polynomial function of $s^{1/2}$ with non-negative coefficients of order at least $3$:
\be\label{E:genericpolynomial}
P(s)=\sum_{i=0}^m c_{i}s^{\frac{3+i}{2}}, \quad c_i\ge0, \ \ m\in\N_0.
\ee
The Einstein summation convention is employed, indicating   summation over repeated indices.
%
%
\subsection{The initial domain $\Omega$ and the harmonic gauge}  
For our initial domain $\Omega$ we choose a simply connected domain $\Omega \subset \R^d$,
where the boundary $\partial\Omega$ will be denoted by $\Gamma.$ We further assume, without loss of generality, that the origin is contained in $\Omega,$
i.e. $0\in\Omega.$
We transform the Stefan problem~(\ref{E:stefan}) set on the moving domain $\Omega(t)$,  to an
equivalent problem on the fixed domain $\Omega$; to do so, we
use a system of {\it harmonic coordinates}, also known as the harmonic gauge or Arbitrary Lagrangian Eulerian (ALE) coordinates in fluid mechanics.

The moving domain $\Omega(t)$  will be represented as the image of a time-dependent family of diffeomorphisms
$\Psi(t) :  \Omega \mapsto \Omega(t)$.  
Let $N$ represent the outward pointing unit normal to $\Gamma$ and let $\Gamma(t)$ be given by
\[
\Gamma(t)=\{x\,|\,\,\,x=x_0+h(t,x_0)N,\,\,\,\,x_0\in\Gamma\}.
\]
Assuming that the signed height function $h(t, \cdot)$ is sufficiently regular and $\Gamma(t)$ remains a small graph over $\Gamma$, we can define a diffeomorphism 
$\Psi:\Omega\to\Omega(t)$
as the elliptic extension of the  boundary diffeomorphism $x_0 \mapsto x_0+h(t,x_0)N $, by solving the following Dirichlet problem:
\begin{align} 
\Delta\Psi&=0 \ \text{ in } \ \Omega,\label{E:gauge}\\
\Psi(t,x)&=x+h(t,x)N(x) \ \,\,\,\, x\in\Gamma.\notag
\end{align} 

We introduce the following new variables set on
the fixed domain $\Omega$:
\begin{alignat*}{2} 
          q&= p\circ \Psi\ && \text{(\textit{temperature})},\\
          v&=-\nabla p\circ\Psi\ && \text{(\textit{``velocity"})},\\
	A&= [D\Psi]^{-1}\ &&\text{(\textit{inverse of the deformation tensor})},\\
	J&= \det D\Psi\ &\quad&\text{(\textit{Jacobian determinant})},
\end{alignat*}
We now pull-back the Stefan problem~\eqref{E:stefan} from $\Omega(t)$  onto the fixed domain $\Omega.$
If we let $g$ denote the Jacobian of the transformation $\Psi(t,\cdot)|_{\Gamma}:\Gamma\to\Gamma(t),$ and let
$n(t,\cdot )$ denote the outward-pointing  unit normal  vector  to the moving surface $\Gamma(t),$  then 
the following relationship holds~\cite{CoSh10}:
\[
J^{-1}\sqrt{g}\,\,n_i\circ\Psi(t,x)=A^k_i(t,x) N_k(x).
\]
It thus follows that
the outward-pointing  unit normal  vector $n(t,\cdot )$ to the moving surface $\Gamma(t)$ can be written as
$
(n \circ \Psi)(t, x) =A^TN/|A^TN| \,.
$
We shall henceforth drop the explicit composition with the diffeomorphism $\Psi$, and simply write
$$ n(t, x) =A^TN/|A^TN| $$ for the unit normal to the moving boundary at the point $\Psi(t, x) \in \Gamma(t)$.

The classical Stefan problem on  the fixed domain $\Omega$ 
is written as (see~\cite{HaSh,HaSh2})
\begin{subequations}
\label{E:ALE}
\begin{alignat}{2}
q_t-A^j_i(A^k_iq,_k),_j&=-v\cdot \Psi_t&& \ \text{ in } \ [0,T) \times \Omega \,,\label{E:ALEheat}\\
v^i+A^k_iq,_k&=0&&\ \text{ in } \ [0,T) \times \Omega\,,  \label{E:ALEv} \\
q&=0&& \ \text{ on } \  [0,T) \times \Gamma \,,\label{E:ALEdirichlet}\\
h_t &= \frac{v \cdot A^TN}{N \cdot A^TN} && \ \text{ on } \  [0,T) \times \Gamma \,,  \label{E:ALEneumann2}\\
\Delta \Psi &=0&& \ \text{ on } \ [0,T) \times \Omega\,,\label{E:Psi_elliptic}\\
 \Psi &=e + h N&& \ \text{ on } \ [0,T) \times \Gamma \,,\label{E:Psi_laplace}\\
q&=q_0>0&&\ \text{ on } \ \{t=0\} \times \Omega   \,,\label{E:ALEinitial} \\
h&=0&&\ \text{ on } \ \{t=0\} \times \Gamma  \,,\label{E:ALEinitialh}
\end{alignat}
\end{subequations}

Problem~(\ref{E:ALE}) is a  reformulation of the problem~(\ref{E:stefan}). 
Observe that the boundary condition~(\ref{E:ALEneumann2}) is equivalent to
\be\label{E:ALEneumann}
\Psi_t\cdot {n}(t)=v\cdot {n}(t) \ \text{ on } \ [0,T) \times \Gamma  \ \text{ so that } \ \Psi(t)(\Gamma) = \Gamma(t) \,,
\ee
which is but a restatement of the Stefan condition~\eqref{E:neumann}. Since the factor $N\cdot A^TN$ will show up repeatedly in various calculations, it is useful to introduce the abbreviation:
\be\label{E:thefactor}
\Lambda\eqdef N\cdot A^TN.
\ee
Note that initially $\Lambda=1$ and it will remain close to $1$, since for small $h$ the transition matrix $A$ remains close to the identity matrix.

Since the identity map $e: \Omega \to \Omega$ is harmonic in $\Omega$ and $\Psi -e = h N$ on $\Gamma$, 
standard elliptic regularity theory for solutions to (\ref{E:gauge}) shows that for $t \in [0,T)$,
\begin{equation*}\label{DM}
\|\Psi(t, \cdot) -e\|_{H^s(\Omega)}\le C\|h(t, \cdot )\|_{H^{s-0.5}(\Gamma)}\,, \ s>0.5,
\end{equation*} 
so that for $h$ sufficiently small and $s$ large enough,  the Sobolev embedding theorem shows that $\nabla\Psi$ is close to $\mbox{Id}$, and by the inverse function theorem,
$\Psi$ is a diffeomorphism.  

\subsubsection{The high-order energy and the high-order norm}\label{S:norms}
We will specialize to the case $d=2$ for the remainder of this paper. The case $d=3$ requires only our norms to contain one more degree of differentiability, while the rest of the argument is entirely analogous.

To define the natural energies associated with the main problem, we must employ tangential derivatives in a neighborhood which is sufficiently close 
to the boundary $\Gamma.$
Near $\Gamma=\partial \Omega $,  it is convenient to use tangential derivatives $\t^{\vec{\alpha}}$, 
while away from the boundary, Cartesian partial derivatives $\partial^{x_i}$ are
natural. 
For this reason,  we introduce a  non-negative $C^{\infty}$  cut-off function $\mu:\bar{\Omega}\to\R_+$ with the
property
\[
\mu(x)\equiv0\,\,\,\,\text{ if }\,\,|x|\leq\rho;\qquad\mu(x)\equiv1\,\,\,\,\text{ if }\,\,\text{dist}(x,\Gamma)\le\sigma.
\]
Here $\rho,\sigma\in\R^+$ are chosen in such a way that $B_{\rho}(0)\Subset\Omega$ and $\{x|\,\text{dist}(x,\Gamma)\le\sigma\}\in\Omega\setminus B_{\rho}(0).$ 
\begin{definition}[Higher-order energies]\label{D:highorderenergies}
The following high-order energy and dissipation functionals are fundamental to our analysis:
\begin{align}
&\E(t)=\E(q,h)(t)\eqdef  \nonumber \\
& \frac{1}{2}\sum_{|\vec{\alpha}|+2b\leq5}\|\mu^{1/2}\t^{\vec{\alpha}}\partial_t^bv\|_{L^2_x}^2
+\frac{1}{2}\sum_{|\vec{\alpha}|+2b\leq6}|(-\g_Nq)^{1/2}\Lambda\t^{\vec{\alpha}}\g_t^b\Psi|_{L^2_x}^2
+\frac{1}{2}\sum_{|\vec{\alpha}|+2b\leq6}\|\mu^{1/2}(\t^{\vec{\alpha}}\partial_t^bq+\t^{\vec{\alpha}}\partial_t^b\Psi\cdot v)\|_{L^2_x}^2 \nonumber \\
& \sum_{|\vec{\alpha}|+2b\leq5}\|(1-\mu)^{1/2}\g^{\vec{\alpha}}\partial_t^bv\|_{L^2_x}^2
+\frac{1}{2}\sum_{|\vec{\alpha}|+2b\leq6}\|(1-\mu)^{1/2}(\g^{\vec{\alpha}}\partial_t^bq+\g^{\vec{\alpha}}\partial_t^b\Psi\cdot v)\|_{L^2_x}^2 \nonumber
\end{align}
and
\begin{align}
& \D(t)=\D(q,h)(t)\eqdef  \nonumber \\
& \sum_{|\vec{\alpha}|+2b\leq6}\|\mu^{1/2}\t^{\vec{\alpha}}\partial_t^bv\|_{L^2_x}^2
+\sum_{|\vec{\alpha}|+2b\leq 5}|(-\g_Nq)^{1/2}\Lambda\t^{\vec{\alpha}}\g_t^b\Psi_t|_{L^2_x}^2
+\sum_{|\vec{\alpha}|+2b\leq 5}\|\mu^{1/2}(\t^{\vec{\alpha}}\partial_t^bq_t+\t^{\vec{\alpha}}\partial_t^b\Psi_t\cdot v)\|_{L^2_x}^2 \nonumber \\ 
& +\sum_{|\vec{\alpha}|+2b\leq6}\|(1-\mu)^{1/2}\g^{\vec{\alpha}}\partial_t^bv\|_{L^2_x}^2
+\sum_{|\vec{\alpha}|+2b\leq5}\|(1-\mu)^{1/2}(\g^{\vec{\alpha}}\partial_t^bq_t+
\g^{\vec{\alpha}}\partial_t^b\Psi_t\cdot v)\|_{L^2_x}^2\,,  \nonumber
\end{align}
where we recall the definition of $ \Lambda $ given in \eqref{E:thefactor}.
Finally, we introduce the total energy $E(t):$
\be
E(t)\eqdef  \sup_{0\le s\le t}\E(\tau) + \int_0^t\D(\tau)\,d\tau.
\ee
\end{definition}

Note that the boundary norms of the gauge function $\Psi$ are weighted by
$\sqrt{-\g_N q}$.
We thus introduce the time-dependent function 
\[
\chi(t)\eqdef \inf_{x\in\Gamma}(-\g_Nq)(t,x)>0,
\]
which will be used to track the weighted behavior of $h$.
It is important to note, that due to the smoothness assumption on $\Gamma$ it is easy to see that for any local coordinate chart $(\g_{s_1},\dots,\g_{s_{d-1}})$ for $\Gamma$ we have the equivalence
\be\label{E:EQUIVALENCEONGAMMA}
\sum_{|\vec{\alpha}|+2b\leq6}|(-\g_Nq)^{1/2}\Lambda\t^{\vec{\alpha}}\g_t^b\Psi|_{L^2_x}^2\approx \sum_{\beta=(\beta_1,\dots,\beta_{d-1}) \atop |\beta|+2b\le 6} |(-\g_Nq)^{1/2}\g_{s_1}^{\beta_1}\dots\g_{s_{d-1}}^{\beta_{d-1}}h|_{L^2(\Gamma)}^2,
\ee
where $X\approx Y$ means that there exist positive constants $C_1$ and $C_2$ such that $C_1 Y \le X\le C_2 Y.$ In our case, the two constants depend on the choice of the local chart. 

\begin{definition}[High-order norm]
The following high-order norm is fundamental to our analysis:
\begin{align} \label{E:highordernorm}
S(t)& \eqdef \sum_{l=0}^3\|\g_t^lq\|_{L^{\infty}H^{6-2l}}^2
+ \|q\|_{L^2H^{6.5}}^2\notag
+\sum_{l=0}^{2}\|\g_t^lq_t\|_{L^2H^{5-2l}}^2 \\
& \ \ + \sup_{0\le s\le t}e^{\beta s}\|q(s,\cdot)\|_{H^5}^2+ \sum_{|\vec{\alpha}|+2l\le6}\|\t^{\vec{\alpha}}\g_t^lv\|_{L^2L^2}^2 \notag\\
& \ \ +\chi(t)\sum_{l=0}^3|\g_t^lh|_{L^{\infty}H^{6-2l}}^2 
+\chi(t)\sum_{l=0}^2|\g_{t}^{l+1}h|_{L^2H^{5-2l}}^2
+|h|_{L^{\infty}H^{4.5}}^4  \notag \\
\end{align}
Here $\beta = 2\lambda - \eta,$ where $\lambda$ is the smallest eigenvalue of the Dirichlet-Laplacian on $\Omega$ and $\eta>0$ is a small but fixed number to be determined later.
\end{definition}

\begin{remark}
A subtle feature of the above definition is the {\bf loss of a $ {\frac{1}{2}} $-derivative}-phenomenon for the temperature $q.$ By the parabolic scaling 
(where one time derivative scales like two spatial derivatives), one might expect $q$ to belong to $L^2H^7([0,T);\Omega),$ since 
$\g_t^{l+1}q\in L^2H^{5-2l}([0,T);\Omega),$ for $l=0,1,2.$ This is, however, not the case, as the height-evolution equation~\eqref{E:ALEneumann2} 
scales  in a hyperbolic fashion, and thus places a restriction on the top-order regularity of the unknown $q,$ allowing only for 
 $q\in L^2H^{6.5}([0,T);\Omega).$
\end{remark}

\subsection{Steady states}

Note that any $C^1$ simply connected domain represents a steady state of~\eqref{E:stefan}. In other words, for any simply connected domain $\bar{\Omega}\in C^1$, the pair
$(\bar{u}\equiv 0, \bar{\Gamma}=\partial\bar{\Omega})$ forms a time-independent solution to~\eqref{E:stefan}. 
In particular, it is challenging to determine {\bf which}  steady state a small perturbation will decay to.
 Thus the problem of asymptotic stability, rather than the optimal regularity of weak/viscosity solutions, is one of the main motivating questions for this work. In particular, we work with classical solutions with a high degree of differentiability on the initial data.

\subsection{Rayleigh-Taylor sign condition or non-degeneracy condition on $q_0$}
With respect to $q_0$, condition (\ref{E:tay}) becomes 
\[
\inf_{x\in\Gamma}[-\g_Nq_0(x)]  \ge \delta > 0 \  \ \text{ on } \Gamma.
\]
For initial temperature distributions that are not necessarily strictly positive in $\Omega$, this condition was shown to be sufficient
for local well-posedness for (\ref{E:stefan}) (see \cite{HaSh, Me, PrSaSi}).  On the other hand, if we require strict
 positivity of our initial temperature function\footnote{Condition (\ref{E:positive}) is natural, since it determines the phase: $\Omega(t) = \{ q(t) >0\}$.}, 
\be\label{E:positive}
q_0>0\quad\text{ in }\Omega\,,
\ee
then the parabolic Hopf lemma (see, for example, \cite{Fr64})  guarantees that $-\g_N q(t,x)>0$ for $0<t<T$ on some a priori (possibly small)
time interval, which, in turn, shows that $\E$ and $\D$ are norms for $t>0$, but uniformity may be lost as $t\to 0$.  To ensure a uniform lower-bound
for $-\g_N q(t)$ as $t\to 0$,  we impose the Rayleigh-Taylor sign condition with the following lower-bound:
\be\label{E:WLOG}
-\g_Nq_0 \ge C_* \int_\Omega q_0 \, \varphi_1 dx \,,
\ee
Here, $\varphi_1$ is the positive first eigenfunction of the Dirichlet Laplacian $-\Delta$ on $\Omega$, and $C_*>0$ denotes a {\em universal} constant. 
The uniform
lower-bound in (\ref{E:WLOG}) thus ensures that our solutions are continuous in time; moreover, (\ref{E:WLOG}) allows us to  establish a 
time-dependent {\it optimal lower-bound} for the quantity
$
\chi(t)=\inf_{x\in\Gamma}(-\g_Nq)(t,x) >0
$
for all time $t\geq0$, which is crucial for our analysis.

\subsection{Main result}
Our main result is a global-in-time stability theorem for solutions of
the classical Stefan problem  for surfaces which are assumed to be close to a given sufficiently smooth domain $\Omega$
and for temperature fields close to zero. 
The notions
of near and close are measured by our energy norms as well as the dimensionless quantity
\be\label{E:k}
K\eqdef \frac{\| q_0\|_4}{\|q_0\|_0}.
\ee
as expressed in the following 
\begin{theorem}\label{T:main}
Let $(q_0,h_0)$ satisfy the 
Rayleigh-Taylor sign condition~(\ref{E:WLOG}), the strict positivity 
assumption~(\ref{E:positive}), and suitable compatibility conditions. 
Let $K$ be defined as in~(\ref{E:k}).
Then there exists an $\epsilon_0>0$ and a monotonically 
increasing function $F:(1, \infty )\to\R_+$, such that if
\be\label{E:smallassum}
S(0)<\frac{\epsilon_0}{F(K)},
\ee
then there exist unique solutions $(q,h)$ to  problem~(\ref{E:ALE})
satisfying
\[
S(t)<C\epsilon_0, \ \ t\in[0,\infty),
\]
for some universal constant $C>0$.
Moreover, the temperature $q(t) \to 0$ as $t \to \infty $ with bound
\[
\|q( t,\cdot )\|_{H^4(\Omega)}^2\leq C e^{-\beta t},
\]
where $\beta=2\lambda -O(\epsilon_0)$ and $\lambda$ is the smallest eigenvalue of the 
Dirichlet-Laplacian on $\Omega.$
The moving boundary $\Gamma(t)$ settles asymptotically to some nearby steady surface 
$\bar{\Gamma}$ and we have the uniform-in-time estimate
\[
\sup_{0\leq t<\infty}|h(t,\cdot )-h_0|_{4.5} \lessim \sqrt{ \epsilon _0}
\]
\end{theorem}

\begin{remark}\label{R:fk}
The increasing function $F(K)$ given in (\ref{E:smallassum}) has an explicit form.  For universal constants $\bar{C},C>1$
chosen in Section~\ref{S:proofofthemaintheorem}, 
\be\label{E:F}
F(K)\eqdef \max\{8K^{2C\bar{C}K^2},\bar{C}^{10}(\ln K)^{10}K^{20\bar{C}\lambda}\}.
\ee
\end{remark}
\begin{remark} The use of 
the constant $K$  in our smallness assumption~(\ref{E:smallassum}) allows us to determine a time $T = T_K$
when the dynamics of the Stefan problem become strongly dominated by the  projection of  $q$ onto the first eigenfunction $\varphi_1$
of the Dirichlet-Laplacian.
Explicit knowledge of the $K$-dependence in the smallness assumption~(\ref{E:smallassum}) permits the use of energy estimates 
to show that solutions exist in 
our energy space on the time-interval $[0,T_K]$.  For $t\ge T_K$, 
certain error terms (that cannot be controlled by our norms for large $t$) become \underline{sign-definite with a good sign}.
\end{remark}

\begin{remark}
An analogous theorem was stated in~\cite{HaSh2}, for perturbations of steady surfaces initially close to a sphere. Therefore, this work generalizes that result. Moreover, our methods are general enough to apply to other geometries as well.
An example is that of a free boundary parametrized as a graph over a periodic flat interface.
\end{remark}

\begin{remark}[On compatibility conditions]
The first compatibility condition on the initial temperature $q_0$ is 
\[
q_0|_{\Gamma}=0.
\]
The second condition arises by restricting the parabolic equation~\eqref{E:ALEheat} to the boundary $\Gamma$ and using the boundary conditions~\eqref{E:ALEdirichlet} and~\eqref{E:ALEneumann}. It gives
\[
\g_{NN}q_0+(d-1)\kappa_{\Gamma}\g_Nq_0+(\g_Nq_0)^2 = 0 \ \text{ on } \ \Gamma.
\]
Here $\kappa_{\Gamma}$ stands for the mean curvature of $\Gamma.$ 
Higher order compatibility conditions arise by taking time derivatives of~\eqref{E:ALEheat}, re-expressing them in terms of purely spatial derivatives via~\eqref{E:ALEheat} and restricting the resulting equation to the boundary $\Gamma$ at time $t=0.$
\end{remark}
\begin{remark}
An interesting problem is to determine the asymptotic attractor - the steady state $\bar{\Gamma}$ just from the initial data $(u_0,\Gamma_0).$ This is strongly connected to the so-called momentum problem, which is a problem of determining 
the domain $\Omega$ from the knowledge of its harmonic momenta $c_\phi = \int_\Omega\phi\,dx,$ $\phi:\R^d\to\R, \ \Delta\phi=0.$ A related question arises in the Hele-Shaw problem, see~\cite{GuVa}.
\end{remark}
\subsection{Local well-posedness theories}
In~\cite{HaSh}, we established the local-in-time existence, uniqueness, and regularity for the classical Stefan problem  in $L^2$-based Sobolev spaces, without derivative loss, using
 the functional framework given by 
Definition~\ref{D:highorderenergies}.
This framework is natural,  and  relies on the geometric control of the free-boundary,  analogous to that
 used in the analysis of the free-boundary incompressible Euler equations in~\cite{CoSh07,CoSh10}; the second-fundamental form  is controlled by a
 a natural coercive quadratic form,  generated from the inner-product of the tangential derivative 
 of the cofactor matrix $JA$, and the tangential derivative of the velocity of the moving boundary, and
yields control of the norm
$
\int_{\Gamma}(-\g_Nq(t))|\t^k \, h|^2\,dx' 
$ for any $k\ge 3$. 
The Hopf lemma ensures positivity of $-\g_N q(t)$ and the Taylor sign condition
 on $q_0$ ensures a uniform lower-bound as $t \to 0$.

The first local existence results of classical solutions for the classical Stefan problem were established by 
Meirmanov (see~\cite{Me} and references therein) and Hanzawa~\cite{Ha}. 
Meirmanov regularized the problem by adding artificial viscosity to~(\ref{E:neumann}) and fixed
the moving domain by switching to the so-called von Mises variables, obtaining solutions with less
Sobolev-regularity than the  initial data.  Similarly, Hanzawa
used  Nash-Moser iteration  to construct a local-in-time solution, but again, 
 with derivative loss.
A local-in-time existence result for the one-phase multi-dimensional Stefan problem was proved  in~\cite{FrSo}, 
using $L^p$-type Sobolev spaces.   For the two-phase
Stefan problem,  a local-in-time existence result for classical solutions  was established in~\cite{PrSaSi} in the framework
of $L^p$-maximal regularity theory.

\subsection{Prior work}
There is a large amount of literature on the classical one-phase Stefan problem. For an  
overview we refer the reader to~\cite{Fr82,Me,Vi} as well as the introduction to~\cite{HaSh2}.
First,  {\em weak} solutions were defined in~\cite{Ka, Fr68, LaSoUr}.
For the one-phase problem studied herein, a variational formulation
was introduced in~\cite{FrKi75}, wherein additional regularity results for the
free surface were obtained.
In~\cite{Ca77} it was shown that in some space-time neighborhood of points $x_0$ on the free-boundary  that have  Lebesgue density,
the boundary is  $C^1$ in both space and time, and  second derivatives of temperature
are continuous up to the boundary.  Under some regularity assumptions on the temperature, Lipschitz regularity of the free boundary was shown in~\cite{Ca78}.
In related works~\cite{KiNi77, KiNi78} it was shown that the free boundary is analytic in space and of second Gevrey class in time,
under the a priori assumption that the free boundary is $C^1$  with certain assumptions on the temperature function.
In~\cite{CaFr79} the continuity of the temperature was proved in $d$ dimensions.
As for the two-phase classical Stefan problem, the continuity of the temperature in $d$ dimensions for  weak solutions
was shown in~\cite{CaEv}.

Since the Stefan problem satisfies a maximum principle,
its analysis is ideally suited to  another type of weak solution called the {\it viscosity solution}. Regularity of
viscosity solutions for the two-phase Stefan problem was established in a 
series of seminal papers~\cite{AtCaSa1, AtCaSa2}. 
Existence of viscosity solutions for the one-phase problem was established in~\cite{Ki}, and for the
two-phase problem in~\cite{KiPo}.  A local-in-time
regularity result was established in~\cite{ChKi}, where it was shown that initially Lipschitz free-boundaries become
$C^1$ over a possibly smaller spatial region. 
For  an exhaustive 
overview and introduction to the regularity theory of viscosity solutions we refer the reader 
to~\cite{CaSa}. 
In~\cite{Ko98} the author showed by the use of von Mises variables and 
harmonic analysis, that an priori $C^1$ free-boundary in the two-phase problem becomes smooth.

In order to understand the asymptotic behavior of the classical Stefan problem on  {\em external} domains, in~\cite{QuVa} the authors
proved that on a complement of a given bounded domain $G$, with {\em non-zero} boundary conditions on the fixed boundary $\g G$, the solution
to the classical Stefan problem converges,  in a suitable sense,  to the corresponding solution of the Hele-Shaw problem and sharp global-in-time expansion rates 
for the expanding liquid blob are obtained. Moreover,  the blob asymptotically has the geometry of a ball.
Note that the non-zero boundary conditions act as an effective {\em forcing} which is absent from our problem and the 
techniques of~\cite{QuVa} do not  directly apply.  Since the corresponding Hele-Shaw problem (in the absence of surface tension and forcing) is {\em not}
a dynamic problem, possessing only time-independent solutions,  we are not able to use the  Hele-Shaw solution as a  comparison problem for our problem.

A global stability result for the two-phase classical Stefan problem in a smooth
functional framework was also established in~\cite{Me} for a specific (and somewhat restrictive) perturbation
of a flat interface, wherein the initial geometry is a strip with imposed Dirichlet temperature conditions on the fixed top and
bottom boundaries, allowing for only one equilibrium solution.
A global existence result for smooth solutions was given in~\cite{DaLe} under the log-concavity assumption
on the initial temperature function, which in light of the level-set reformulation of the Stefan problem, requires convexity of the
initial domain (a property that is preserved by the dynamics).

\begin{remark}
We remark that  global stability of solutions in  the {\em presence} of surface tension does not require the use of function framework with a  decaying weight, such as $-\g_Nq(t)$. In this regard, the surface tension problem is
simpler for two important reasons:  first, the surface tension contributes a positive-definite energy-contribution that is uniform-in-time, and provides 
better regularity of the free-boundary (by one spatial derivative), and second, the space of equilibria is finite-dimensional and thus it is easier to understand the degrees-of-freedom that determine the asymptotic state of the system. 
\end{remark}

\subsection{Methodology}

Broadly speaking, our methods combine high-order energy estimates with maximum principle techniques. Once the problem is formulated on the fixed domain with the help of the harmonic gauge explained above, we notice that the natural quadratic energy quantities that track the regularity behavior of the moving boundary, come weighted with the normal derivative of the temperature. This weight is a time-dependent quantity and its evolution is tied to the free boundary itself. This coupling is nonlinear and it is one of the central difficulties in closing our estimates.

Our strategy is based on~\cite{HaSh2} and it contains three basic steps. We first show that under the assumption of smallness on the norm $S(t)$ over some time interval $[0,T]$,
the energy $E$ and the norm $S$ are equivalent, i.e.
\be\label{E:equivalenceintro}
S(t) \lesssim E(t) \lesssim S(t), \quad t\in [0,T].
\ee
Our second step is to establish the key energy inequality in the form
\be\label{E:basicinequalityintro}
E(t) \le C_0 + \frac{1}{2}\sum_{|\vec{\alpha}|+2l\le 6}\int_0^t\int_{\Gamma}(\g_Nq_t)|\t^{\vec{\alpha}}\g_t^lh|^2\,dS(\Gamma) +  P(S(t)),
\ee
where $P$ is a cubic polynomial (see~\eqref{E:genericpolynomial}) and $C_0$ is a small quantity depending only on the initial data. Combining~\eqref{E:equivalenceintro} and~\eqref{E:basicinequalityintro}, we infer that
\be\label{E:basicinequalityintro2}
S(t) \le \tilde{C}_0 +  C\underbrace{\sum_{|\vec{\alpha}|+2l\le 6}\int_0^t\int_{\Gamma}(\g_Nq_t)|\t^{\vec{\alpha}}\g_t^lh|^2\,dS(\Gamma)}_{\text{ dangerous term }}
 +  P(S(t))
\ee
on the time interval of existence. 
If it were not for the sum on the right-hand side above, a simple continuity argument would yield a global existence result for small initial data.  However, the sum appearing on the right-hand side of~\eqref{E:basicinequalityintro2}, while seemingly cubic, cannot be bounded by $P(S(t))$. 
Instead, in the third step we show that after a certain, precisely quantified amount of time, this ``dangerous term" becomes negative and can thus be trivially bounded from above by zero. 

The key novelty with respect to~\cite{HaSh2} is a new quantitative lower bound on the weight $-\g_Nq$ which appears in our definition of the energy $E(t).$ Note that  this quantity is expected to converge exponentially fast to $0$ as the 
unknowns settle to an asymptotic equilibrium. We employ the theory of ``halfeigenvalues" associated with the Bellman-Pucci-type operators to generate a comparison function, which then allows us to use the maximum principle and get a nearly sharp lower bound:
\[
-\g_Nq \gtrsim e^{(-\lambda+O(\epsilon))t},
\]
where $\lambda$ denotes the first Dirichlet eigenvalue associated with the domain $\Omega.$ In our previous work~\cite{HaSh2}, we relied on a rather explicit Bessel-type comparison functions used by Oddson in~\cite{Od}, which in particular, required that we work in a nearly spherical domain. The above lower bound is much more flexible and it is explained carefully in Section~\ref{S:pucci}.

The presentation in the paper is considerably simplified with respect to~\cite{HaSh2} and we believe that our energy method in conjunction with maximum principles can be useful for the stability analysis in other free boundary problems in absence of surface tension.
 
\subsection{Plan of the paper}
In Section~\ref{S:bootstrapassumptions}, we introduce the bootstrap assumptions and formulate the equivalence relationship between the energy and the norm. In Section~\ref{S:pucci} we provide a dynamic lower bound estimate on $\chi(t)$. This is the main new ingredient with respect to~\cite{HaSh2} and we use the theory of half-eigenvalues for the Pucci operators. Finally, in Section~\ref{S:proofofthemaintheorem}, we give the proof of Theorem~\ref{T:main}, thereby explaining our continuity method as well as a comparison argument used to show the sign-definiteness of the ``dangerous linear terms" described  above.

\section{Bootstrap assumptions and norm-energy equivalence}\label{S:bootstrapassumptions}

\subsection{The bootstrap assumptions}

Let $[0,T)$ be a given time-interval of existence of solutions to~\eqref{E:ALE}. We assume that the following two assumptions hold:
\begin{align}
S(t) &\le \epsilon, \ \ t\in[0,T), \label{E:bootstrap1}\\
\chi(t) & \gtrsim c_1 e^{-(\lambda+\frac{\eta}{2})t}, \ \ t\in [0,T), \label{E:bootstrap2}
\end{align}
where $\epsilon$ and $\eta$ are to be chosen sufficiently small later and $\lambda$ stands for the first Dirichlet eigenvalue associated with the domain $\Omega.$

\subsection{Norm $S$ and total energy $E$ are equivalent}

Recall the notation ``$\approx$" introduced in~\eqref{E:EQUIVALENCEONGAMMA}.
\begin{proposition}\label{P:equivalencenormenergy}
There exists a sufficiently small $\epsilon'$ such that if $S(t)\le\epsilon'$ on a time interval $[0,T]$
then
\[
S(t) \approx E(t),\quad \forall t\in [0,T].
\]
\begin{proof}
The proof of this fact is one of the pillars of our strategy. It has been presented in detail in Sections 2.1 - 2.5 and Section 4.2 of~\cite{HaSh2} and, therefore,
 we omit it here. 
We note that the direction $S(t) \lesssim E(t)$ is obviously harder to prove, as the energy function $E(t)$ a-priori controls only tangential derivatives of the temperature $q.$ In~\cite{HaSh2} we use a version of the elliptic regularity statement for equations with Sobolev-class coefficients to obtain control of normal derivatives (see \cite{ChSh14}).
\end{proof}
\end{proposition}

\section{Lower bound on $\chi(t)$ and improvement of the second bootstrap assumption}\label{S:pucci}

The heat equation (\ref{E:ALEheat})  for $q$ can be written in non-divergence form as
\begin{subequations}
\label{E:qparabolic}
 \begin{alignat}{2}
q_t-a_{kj}q,_{kj}-b_kq,_k&=0&&\,\text{ in }\,\Omega,\label{E:parabolic}\\
q&=0&&\,\text{ on }\,\Gamma,\label{E:dirichletc}\\
q(0,\cdot)&=&&q_0>0\,\text{ in }\,\Omega
\end{alignat}
\end{subequations}
where the coefficient matrix $a=(a_{kj})_{k,j=1,2}$, and the vector $b=(b_1,b_2)$ are explicitly given by
\be\label{E:coeff}
a_{kj}\eqdef A^k_iA^j_i;\quad b_k\eqdef A^k_{i,j}A^j_i+A^k_i\Psi^i_t.
\ee

By the bootstrap assumption~\eqref{E:bootstrap1} and the definition~\eqref{E:highordernorm} of $S(t),$ we have that $|h|_{4.5}\lesssim \sqrt{\epsilon}$ on $[0,T)$, and therefore by the Sobolev embedding $H^{1}(\Gamma)\hookrightarrow L^{\infty}(\Gamma)$, we infer that 
$|h|_{W^{3,\infty}}\lesssim \sqrt{\epsilon}.$ From this observation,~\eqref{E:coeff}, and the definition of the transition matrix $A$, we infer that
\begin{align*}
|a_{kj}-\delta_{kj}| & \lesssim \sqrt{\epsilon}, \ \ (k,j=1,2), \\
|b_i| & \lesssim \sqrt{\epsilon}, \ \ (i=1,2).
\end{align*}
Therefore, there exists a constant $K>0$ such that the ellipticity constants associated with the matrix $(a_{ij})_{i,j=1,2}$
are between
the values $\mu_1'=1-\frac{K}{2}\sqrt{\epsilon}$ and $\mu_2'=1+\frac{K}{2}\sqrt{\epsilon}$ uniformly over $[0,T).$

Before we proceed with calculating a lower bound for $\chi(t),$ we briefly explain the Bellman operators \cite{Ar, BuEsQu, FeQuSi,GiTr,Li} which are closely connected to the 
well-known extremal Pucci operators. They will allow us to formulate a nonlinear analogue of the ``first" eigenvalue for the elliptic part of the operator defined in~\eqref{E:parabolic}. 

Let $\Omega$ be an arbitrary simply connected $C^1$-domain.
We define the extremal Pucci operator $\mathcal{M}_{\mu_1,\mu_2}^{-}$~\cite{GiTr,BuEsQu} with
parameters $0<\mu_1\le\mu_2$ by
\be\label{E:pucciminus}
\mathcal{M}_{\mu_1,\mu_2}^{-}\varphi(x) \eqdef  \inf_{\mathcal{L} \in \mathcal{K}_{\mu_1,\mu_2}} \mathcal{L} \varphi(x).
\ee
Here $\mathcal{K}_{\mu_1,\mu_2}$ denotes the set of all linear second-order elliptic operators, whose ellipticity constant is between $\mu_1$ and $\mu_2$, i.e., 
\begin{align}
\mathcal{K}_{\mu_1,\mu_2} \eqdef  \big\{L| \ \ & L=a_{ij}\g_{ij}+b_i\g_i+c, \ \ a_{ij},\,b_i,\, c\in C^0(\Omega), \\
& \mu_1|\xi|^2 \le a_{ij}\xi_i\xi_j \le \mu_2 |\xi|^2, \ \ \xi\in\R^d\big\}. \notag
\end{align}

It is well known that the operators $\mathcal{M}^{-}_{\mu_1,\mu_2}$ are, in general,  fully nonlinear second-order elliptic operators, positive, and homogenous of order one. The latter property allows us to formulate an associated ``eigenvalue" problem, looking for the solutions of
\begin{align}\label{E:eigenvalueproblem}
-\mathcal{M}_{\mu_1,\mu_2}^{-} u & = \lambda u \ \ \text{ in } \ \ \Omega, \\
u & = 0 \ \ \text{ on } \ \ \partial\Omega. \notag
\end{align}
We next state some of the results from~\cite{Li} that 
that will play an important role in this paper (for further references on the so-called half-eigenvalues associated with positive homogenous fully nonlinear operators we refer the reader, for example, to~\cite{BuEsQu, Ar, FeQuSi}):

\begin{itemize}

\item There exist two positive constants $\lambda_1$ and $\lambda_2$ called the {\em first half-eigenvalues} 
and two functions $\varrho_1, \varrho_2\in C^2(\Omega)\cap C(\bar{\Omega})$ such that 
$(\lambda_1,\varrho_1)$ and $(\lambda_2,\varrho_2)$ solve~\eqref{E:eigenvalueproblem}, and 
$\varrho_1>0$, $\varrho_2<0$ in $\Omega$.

\item
The first two half-eigenvalues are simple, i.e. all positive solutions to~\eqref{E:eigenvalueproblem} are of the form
$(\lambda_1,\alpha\varrho_1)$ with $\alpha>0$ and analogously, all negative solutions are of the form $(\lambda_2,\alpha\varrho_2)$, $\alpha>0.$

\item
Finally, the first two half-eigenvalues are characterized in the following manner:
\be\label{E:halfeigenvalue}
\lambda_1 = \sup_{A\in\mathcal{K}_{\mu_1,\mu_2}}\mu(A),\qquad \lambda_2 = \inf_{A\in\mathcal{K}_{\mu_1,\mu_2}}\mu(A),
\ee
where $\mu(A)$ stands for the smallest Dirichlet eigenvalue associated with the second order linear elliptic operator $A.$
\end{itemize}

\subsection{Lower bound on $\chi(t)$ and the improvement of~\eqref{E:bootstrap2}}

The key ingredient to the proofs of Propositions~\ref{P:equivalencenormenergy} and~\ref{P:energyestimate} is a quantitative lower bound on the weight $\chi(t).$ This is achieved by using the maximum principle and constructing an appropriate comparison function.

\begin{lemma}\label{L:pucci}
Under the bootstrap assumptions~\eqref{E:bootstrap1} -~\eqref{E:bootstrap2} with $\epsilon$ sufficiently small, the following inequality holds: 
\[
\chi(t) \gtrsim c_1 e^{-(\lambda+\tilde{\lambda}(t))t},
\]
where $c_1=\int_{\Omega}q_0\varphi_1\,dx$ is the first coefficient in the eigenfunction expansion of the initial datum $q_0$ with respect to the $L^2$ orthonormal basis $\{\varphi_1,\varphi_2,\dots\}$ of the eigenvectors of the operator $-\Delta$
on $\Omega$,
i.e $q_0=c_1\varphi_1+c_2\varphi_2+\dots$.
Moreover, $\lambda$ stands for the smallest Dirichlet eigenvalue associated with the domain $\Omega$ and $\tilde{\lambda}(t)$  satisfies the estimate:
\[
|\tilde{\lambda}(t)| \le C\sqrt{\epsilon}.
\]
In particular, with $\epsilon>0$ sufficiently small so that $C\sqrt{\epsilon}<\eta/4$, we obtain the improvement
of the  
bootstrap bound~\eqref{E:bootstrap2} given by $\chi(t)\gtrsim c_1e^{-(\lambda_1+\eta/4)t}.$
\end{lemma}

\begin{proof}
Let us choose $\mu_1\eqdef 1-K\sqrt{\epsilon}$ and $\mu_2\eqdef 1+K\sqrt{\epsilon}.$  Recall that $K$ was defined in the paragraph after~\eqref{E:coeff}. It follows that $L\in\mathcal{K}_{\mu_1,\mu_2}$. We let $\varrho_1$ be the first half-eigenvector associated to $\mathcal{M}_{\mu_1,\mu_2}^-$ as above.
Consider the following comparison function
\[
v(t,x) \eqdef e^{-\lambda_1t}\varrho_1.
\]
Note that $v$ vanishes on $\partial\Omega = \Gamma.$
A straightforward calculation together with the definition of $\mathcal{M}_{\mu_1,\mu_2}^-$ shows that
\begin{align}
(\g_t-L)v &= -\lambda_1v-e^{-\lambda^-_1t}L\varrho_1 \notag \\
&  \le -\lambda_1v - e^{-\lambda_1t}\mathcal{M}_{\mu_1,\mu_2}^-\varrho_1 \notag \\
& = -\lambda_1v+e^{-\lambda_1t}\lambda_1 \varrho_1 \notag \\
& =0.\notag 
\end{align}
Therefore $v$ is a subsolution to the parabolic problem~\eqref{E:qparabolic}. 
The next key observation is that the eigenfunction $\varrho_1(x)$ behaves like a constant multiple of the distance function $\text{dist}(x,\Gamma)$
as $x$ approaches the boundary $\Gamma$. Namely, since the operator $\mathcal{M}^-$ is concave, the solution is $C^{2,\alpha}$~\cite{Sa,CafCab} and 
the Hopf lemma $-\g_N\varrho_1>0$ holds (see for instance Lemma 2.1 in~\cite{BuEsQu}).
Therefore, function $v$
behaves like $c \ \text{dist}(x,\Gamma)e^{-\lambda^-_1t}$ as $x$ approaches the boundary $\Gamma$ for some constant $c.$
Here $\text{dist}(x,\Gamma)$ denotes the distance function to the boundary $\Gamma.$
We first want to show that for any arbitrarily small time $\sigma>0$ there exists a strictly positive constant $\delta(\sigma)>0$ such that
$q-\delta v $ is a {\em positive} supersolution to the parabolic problem~(\ref{E:qparabolic}) on the time interval $[\sigma,T).$

Since $v$ is a subsolution and $q$ is a solution, it follows that for any $\delta>0$, $q-\delta v$ is a supersolution.  The positivity of  $q-\delta v$ at $t= \sigma $ 
follows from the parabolic Hopf lemma, from which we infer the existence of a
 constant $\delta( \sigma )$ such that $\frac{q}{v}>\delta( \sigma )$ uniformly over
$\bar{\Omega}$. Note that we  have used the fact that $v(\sigma,x)$ behaves like $c \times \mbox{dist}(x)$ near the boundary $\Gamma$  for some positive constant $c$. Thus by the maximum principle, $q-\delta(\sigma) v\ge0$ on $[\sigma,T).$
This implies
\[
q(t,x)\ge \delta(\sigma) v(t,x)\ge C\delta(\sigma) \text{dist}(x,\Gamma) e^{-\lambda_1t}, \ \ t\in[\sigma,T),
\]
which yields
\[
-\frac{\g q(t,x)}{\g N} \ge C\delta(\sigma) e^{-\lambda_1t}, \ \ t\in[\sigma,T).
\]
The above estimate is however not yet satisfactory, as the constant $\delta(\sigma)$ may degenerate as $\sigma$ goes to zero.

We now revisit our usage of the parabolic Hopf lemma above.
For small $t>0$ let 
\[
\Omega_{t} = \{x\in\Omega\,\big| \ \text{dist}(x,\Gamma)\ge t\}, \ \ t>0.
\]
Note that $\Omega_t$ is a compact proper subset of $\Omega.$ 
From the proof of the parabolic Hopf lemma (see for instance Theorem 3.14 in~\cite{Fr64}), the 
value $-\partial q/ \partial N|_{t=\sigma}$ 
is proportional to the minimal value of the temperature $q$ on a space-time region strictly contained in the space-time slab 
$K_{t}:=\Omega_{t}\times [t/2,3t/2]\varsubsetneq\Omega\times[0,2t]$ divided by $t$ (which is proportional to the distance of $K_t$ from the parabolic boundary of $\Omega\times[0,2t]$).
Note that, as $t$ approaches $0$ we may loose uniformity-in-time in our constants. This is however not the case since $\g_Nq$ is continuous 
at $t=0$ and by the assumption~\eqref{E:WLOG}
\be\label{E:WLOG2}
-\g_Nq_0=\frac{-\g_Nq_0}{c_1}c_1\ge C_* c_1.
\ee
Assumption~\eqref{E:WLOG} is used only in~\eqref{E:WLOG2} to insure that there exists a universal constant $C_*$ independent of $c_1$ such that  $L=(-\g_Nq_0)/c_1>C_*.$
The quantity $L$ is dimensionless, and the assumption $L>C_*$  is not a restriction on the initial data. In other words, if we had not assumed~\eqref{E:WLOG},  
the only modification in the statement of the main theorem would be that the smallness assumption on initial data~(\ref{E:smallassum}) 
is additionally expressed in terms of $L$ as well.

As to the bound on $\tilde{\lambda}$, note that by~\eqref{E:halfeigenvalue}, the exponent 
$\lambda_1$ is characterized by the condition
\[
\lambda_1=\sup_{A\in \mathcal{K}_{\mu_1,\mu_2}} \mu(A).
\]
Since $|\mu_i-1|\lesssim\sqrt{\epsilon},$ $i=1,2,$ it follows that for any matrix $A\in\mathcal{K}_{\mu_1,\mu_2}$ the estimate $|A-\text{Id}|\lesssim\sqrt{\epsilon}$ holds.
Since the function $\mu(\cdot)$ is a continuous function from the space of $2\times2$ matrices into $\R$,
it thus follows that 
\[
|\tilde{\lambda}|=|\lambda_1-\mu(\text{Id})| = |\sup_{A\in \mathcal{K}_{\mu_1,\mu_2}}\mu(A)-\mu(\mbox{Id})|  \lesssim \sqrt{\epsilon}. 
\]
\end{proof}

\section{Energy estimates and improvement of the first bootstrap assumption}\label{S:proofofthemaintheorem}

\begin{proposition}\label{P:energyestimate}
Assuming the bootstrap assumption~\eqref{E:bootstrap1} and with  $\epsilon>0$ chosen sufficiently small,
\be\label{E:basicinequality}
E(t) \le C_0 + \frac{1}{2}\sum_{|\vec{\alpha}|+2l\le 6}\int_0^t\int_{\Gamma}(-\g_Nq_t)|\t^{\vec{\alpha}}\g_t^l\Psi|^2\,dS(\Gamma) + C P(S(t)),
\ee
where $C_0$ depends only on the initial data,  $C>0$ is a generic positive constant depending only on the dimension $d$,
and $P$ denotes an order-$r$ polynomial with $r\ge 3$ of the form~\eqref{E:genericpolynomial}.
\end{proposition}

\begin{proof}
The proof of the proposition is entirely analogous to the proof of Proposition 3.4 from~\cite{HaSh2}.
\end{proof}

\begin{proposition}\label{P:positive}
Let the solution $(q,h)$ to the Stefan problem~(\ref{E:ALE}) exist on a given maximal interval of existence $[0,\mathcal{T})$ on 
which the bootstrap assumptions~(\ref{E:bootstrap1}) and~(\ref{E:bootstrap2}) are satisfied. 
\begin{enumerate}
\item[(a)]
There exists a universal constant $\bar{C}$ such
that if the smallness assumption~(\ref{E:smallassum}) for the initial data holds and if $\mathcal{T} \ge  T_K\eqdef \bar{C}\ln K$,  then
\[
-q_t(T_K,x)>Cc_1e^{-\lambda_1T_K}\varphi_1(x),\quad x\in B_1(0),
\]
where $\varphi_1$ is the first eigenfunction of the Dirichlet-Laplacian on $\Omega$ and $c_1=\int_{\Omega}q_0\varphi_1\,dx$.
As a consequence,
\[
 \inf_{x\in\Gamma}\g_Nq_t(T_K,x)>0.
\]
\item[(b)]
With the smallness assumption~\eqref{E:smallassum}, we indeed have the bound
$\mathcal{T}\ge \bar{C}\ln K.$
\item[(c)]
Moreover, under the same assumption as in part (b), the following lower bound on $\g_Nq(t,x)$ holds:
\be\label{E:positivity}
\inf_{x\in\Gamma}\g_Nq(t,x)>0, \ \ t\in[T_K,\mathcal{T}).
\ee
\end{enumerate}
\end{proposition}

\begin{proof}
The proof of part (a) of  is  the same as the proof of Lemma 4.2 in~\cite{HaSh}.

As to the proof of part (b), we start by making the claim that the dangerous term from the inequality~\eqref{E:basicinequality} satisfies the bound
\be\label{E:error1}
\Big|\sum_{|\vec{\alpha}|+2l\le 6}\int_0^t\int_{\Gamma}(-\g_Nq_t)|\t^{\vec{\alpha}}\g_t^l\Psi|^2\,dS(\Gamma)\Big|
\le CK^2\int_0^te^{\eta \tau}S(\tau)\,d\tau.
\ee
Note, that if $|\vec{\alpha}|+2l\le6,$ then
$$
\Big|\int_0^t\int_{\Gamma}(-\g_Nq_t) \big|\t^{\vec{\alpha}}\g_t^l\Psi\big|^2\, dS\,d\tau\Big|
=
\Big|\int_0^t\int_{\Gamma}\frac{-\g_Nq_t}{-\g_Nq}(-\g_Nq)\big|\t^{\vec{\alpha}}\g_t^l\Psi\big|^2\Big|\,dS\,d\tau
\le C
\int_0^t\Big|\frac{\g_Nq_t}{\g_Nq}\Big|_{L^{\infty}}S(\tau)\,d\tau.
$$
In order to bound the term $\big|\frac{\g_Nq_t}{\g_Nq}\big|$, we need a decay estimate
for the numerator $|\g_Nq_t|$. 
The Sobolev embedding theory would yield the bound $|\g_Nq_t|_{L^{\infty}}\lesssim\|q_t\|_{2 + \delta }$ for 
$ \delta >0$, but by definition of the norm $S$, it is only the $H^2(\Omega)$-norm of $q_t$ for which
we have the desired decay.  We obtain  the decay estimate for $q_t$ from  Appendix B of \cite{HaSh2}:
\be\label{E:barrier2}
|\g_Nq_t|_{L^{\infty}}\lesssim K^2c_1e^{-\beta t/2}.
\ee
It then follows from the bootstrap assumption~(\ref{E:bootstrap2}) that
\beas
\Big|\frac{\g_Nq_t(\tau)}{\g_Nq(\tau)}\Big|_{L^{\infty}}
\leq \frac{CK^2c_1e^{-(\lambda_1-\eta/2)\tau}}{c_1e^{-(\lambda_1+\eta/2)\tau}}
\leq CK^2e^{\eta \tau} \,,
\eeas
which, in turn,  establishes (\ref{E:error1}). In conjunction with Proposition~\ref{P:energyestimate}, this yields the bound
\be
E(t)\le E(0) + CK^2\int_0^te^{\eta\tau}S(\tau)\,d\tau+ C\epsilon S(t).
\ee
By Proposition~\ref{P:equivalencenormenergy}, with $\epsilon$ sufficiently small,  we conclude that
\be\label{E:gron1}
E(t)\leq2E(0)+CK^2\int_0^{t}e^{\eta \tau}E(\tau)\,d\tau,\quad t\in[0,\mathcal{T}],
\ee
where $\mathcal{T}$ is the maximal interval of existence on which the
bootstrap assumptions~(\ref{E:bootstrap1}) and~(\ref{E:bootstrap2}) hold (with $\epsilon$ sufficiently small).
A straightforward Gronwall-type argument based on~\eqref{E:gron1}, identical to Step 1 of the proof of Theorem 1.2 in~\cite{HaSh},  implies that
as long as the $\eta$ from the bootstrap assumption~\eqref{E:bootstrap2} is smaller than $\bar{C}\ln K$, the maximal interval of existence $[0,\mathcal{T})$, on which both the bootstrap assumptions~\eqref{E:bootstrap1} and~\eqref{E:bootstrap2} are valid, satisfies
$\mathcal{T}>\bar{C}\ln K$, and the following exponentially growing bound holds:
\be\label{E:intermed}
E(t) \le 2E(0)\, e^{CK^2 t}, \ \ t\in[0,\mathcal{T}).
\ee

To prove the part (c), we resort to  maximum principle techniques once again. 
To this end, we define a barrier function $\psi$ to be the solution of the 
following elliptic problem
\begin{align}
\Delta\psi & = -1 \ \ \text{ in } \Omega \label{E:barrierfunction}\\
\psi & = 0 \ \ \text{ on } \Gamma. \notag
\end{align}
We then define the comparison function $\mathcal{F} :[0,\mathcal{T})\times\Omega\to \R$ via
\be\label{E:barrier}\mathcal{F} (t,x)=\kappa_1 e^{-\frac{3}{2}\lambda t}(\varphi_1(x)-\kappa_2\psi),
\ee
with positive constants $\kappa_1,\kappa_2$ to be specified later.
A straightforward calculation shows that
\begin{align}
(\g_t-a_{ij}\g_{ij}-b_i\g_i)\mathcal{F} 
 & =\kappa_1e^{-\frac{3}{2}\lambda t}\big[-\frac{1}{2}\lambda \varphi_1-\kappa_2+\frac{3}{2}\lambda \kappa_2\psi 
 -(a_{ij}-\delta_{ij})(\varphi_1-\kappa_2\psi)-b\cdot(\nabla\varphi_1-\kappa_2\nabla\psi)\big].  \label{E:comp}
\end{align}
Note that
the first and the second terms in the parenthesis on the right-hand side of~(\ref{E:comp}) are negative, while the fourth and the fifth terms,  are 
small, being of order $\epsilon$.
If $x$ is close to $\Gamma$, then the second term dominates the third term and if $x$ is away from the boundary $\Gamma$, then one can choose $\kappa_2>0$
so that the first term dominates the third term. Thereby we use the fact that $\varphi_1$ and $\psi$ both vanish at $\Gamma$, they are both non-negative 
(by the maximum principle), and both satisfy the Hopf lemma (since they are both super-solutions).
It follows, then,  that there exists a $\kappa_2>0$ and some constant $C_1$ such that
\be\label{E:neg}
(\g_t-a_{ij}\g_{ij}-b_i\g_i)  \mathcal{F} <-C_1\kappa_1e^{-\frac{3}{2}\lambda t}.
\ee
It then follows from~(\ref{E:neg}) and~(\ref{E:coeff}) that
\be\label{E:neg1}
(\g_t-a_{ij}\g_{ij}-b_i\g_i)(-q_t- \mathcal{F} )>-(\g_ta_{ij}\,q,_{ij}+\g_tb_{i}\,q_i+\g_tA_{,i}^k\,q,_kw^i+A^k_iq,_kw^i_t)+C_1\kappa_1e^{-\frac{3}{2}\lambda t}.
\ee
Note, however, that the term in parenthesis on the right-hand side above
is a quadratic non-linearity and as such decays at least
as fast as $e^{-\beta t}$:
\begin{align*}
& \|\g_ta_{ij}\,q,_{ij}+\g_tb_{i}\,q_i+\g_tA_{,i}^k\,q,_kw^i+A^k_iq,_kw^i_t\|_{L^{\infty}} 
\leq C_2c_1\epsilon e^{-\beta t}. \label{E:track}
\end{align*}
Now, using~(\ref{E:neg1}) and the above bound, we note that by choosing the constant $\kappa_1\eqdef \frac{C_2}{C_1}c_1\epsilon $, 
we have that
\[
(\g_t-a_{ij}\g_{ij}-b_i\g_i)(-q_t-  \mathcal{F} )>C_2c_1\epsilon e^{-\frac{3}{2}\lambda t}-C_2c_1\epsilon e^{-\beta t}>0,
\]
since $\beta=2\lambda -\eta>\frac{3}{2}\lambda $.
The previous bound implies that $-q_t-  \mathcal{F} $ is a supersolution for the operator $\g_t-a_{ij}\g_{ij}-b_i\g_i$. Moreover,
by the construction of $ \mathcal{F} $, we have $-q_t-  \mathcal{F} =0$ on $\Gamma$.
Furthermore, at time $T_K=\bar{C}\ln K$, we have by the part (b) of the proposition and~(\ref{E:barrier}), that
\[
(-q_t- \mathcal{F} )|_{T=\bar{C}\ln K}>Cc_1e^{-\lambda T}\varphi_1(x)-Cc_1\epsilon e^{-\frac{3}{2}\lambda T}\varphi_1(x)
+Cc_1\epsilon \kappa_2e^{-\frac{3}{2}\lambda T}\psi(x)>0
\]
for $\epsilon$ sufficiently small. Thus, as in the proof of Lemma~\ref{L:pucci}, there exists a constant $m>0$ such that
\[
-q_t(t,x)-  \mathcal{F} (t,x)\geq m \ \text{dist}(x,\Gamma)e^{-(\lambda+O(\epsilon ))t}, \ \ t>T_K,
\]
or, in other words,
\beas
-q_t(t,x)&\geq& m \ \text{dist}(x,\Gamma)e^{-(\lambda +O(\epsilon ))t}+Cc_1\epsilon \, \text{dist}(x,\Gamma)e^{-\frac{3}{2}\lambda t}\left(\frac{\varphi_1(x)}{\text{dist}(x,\Gamma)}-\kappa_2\frac{\psi(x)}{\text{dist}(x,\Gamma)}\right)\\
&=& \text{dist}(x,\Gamma)e^{-(\lambda +O(\epsilon ))t}
\Big(m+Cc_1\epsilon e^{(-\frac{1}{2}\lambda t-O(\epsilon ))t}\left(\frac{\varphi_1(x)}{\text{dist}(x,\Gamma)}-\kappa_2\frac{\psi(x)}{\text{dist}(x,\Gamma)}\right)\Big),
\eeas
which readily gives the positivity of $\g_Nq_t$ on the time-interval $[T_K,\mathcal{T}[$ since $\frac{\varphi_1(x)}{\text{dist}(x,\Gamma)}-\kappa_2\frac{\psi(x)}{\text{dist}(x,\Gamma)}>0$
by our choice of $\kappa_2$ above.
We conclude that the positivity of $-q_t$ at time $T_K=\bar{C}\ln K$ is a property preserved by our bootstrap regime and,
moreover, we obtain a quantitative lower bound on $\g_Nq_t$ on the time interval $[T_K,\mathcal{T}[$.
\end{proof}

\begin{remark}
In the proof of part (b) of Proposition~\ref{P:positive}, we made a rather crude use of the energy estimate given by Proposition~\ref{P:energyestimate}. In particular, we cannot use this argument to prove global existence, as the constants grow in time;
however, in part (c) of the proposition, we have used a more sophisticated argument based on the maximum principle to infer the sign-definiteness 
of the term $\g_Nq_t$ after a fixed amount of time has passed.
\end{remark}

\noindent
{\bf Proof of Theorem~\ref{T:main}.}
Assume for contradiction that $\mathcal{T}<\infty.$
For any $t\in[T_K,\mathcal{T}[$, the energy identity takes the form
\begin{align*}
E(t)+\frac{1}{2}\sum \int_{T_K}^t\int_{\Gamma}\g_Nq_t|\t^{\vec{\alpha}}\g_t^l\Psi|^2\,dS
\le E(T_K)+P(S(t)) \le E(T_K)+ O(\epsilon)E(t).
\end{align*}
Note here the absence of the exponentially growing term in the above bound 
as compared to the inequality~\eqref{E:intermed}. 
This is due to the fact that terms $\int_{T_K}
^t\int_{\Gamma}\g_Nq_t|\t^{\vec{\alpha}}\g_t^l\Psi|^2\,dx$, $|\vec{\alpha}|+2l\le6,$ are positive and no longer treated as error terms.
By absorbing the small multiple of $E(t)$ into the left-hand
side,  and using the positivity of $\g_Nq_t$ from Step 2, we obtain that
\be\label{E:smallness2}
E(t)
\leq 2E(T_K)\leq 8E(0)e^{2CK^2T_K}, \ \ t\in[T_K,\mathcal{T}),
\ee
by~(\ref{E:intermed}).
Finally, we choose $\epsilon_0$ in the statement of Theorem~\ref{T:main} so that $\epsilon_0<\epsilon/2$.
The bound~(\ref{E:smallness2}) and the condition $E(0)\lesssim\epsilon_0/F(K)$ (with $F(K)$ given as in~(\ref{E:F}))
imply
\[
E(t)
\leq \frac{\epsilon}{2}, \ \  t\in[T_K,\mathcal{T}).
\]
Together with Lemma~\ref{L:pucci}, we infer that the bootstrap assumptions~(\ref{E:bootstrap1}) and~(\ref{E:bootstrap2}) are improved.
Since $E(\cdot)$ is continuous in time, we can
extend the solution by the local well-posedness theory to an interval $[0,\mathcal{T}+T^*]$
for some small positive time $T^*$. This however contradicts the maximality of $\mathcal{T}$ if $\mathcal{T}$ were finite and
hence $\mathcal{T}=\infty$.
This concludes the proof of the main theorem.

\section*{Acknowledgements}
 SS was supported by the National Science Foundation under grant DMS-1301380
and by the Royal Society Wolfson Merit Award.  MH was partly supported by the National Science Foundation under grant  DMS-1211517.  Some of this work was completed 
during the program {\it Free Boundary Problems and Related Topics} at the Isaac Newton Institute for Mathematical Sciences at Cambridge, UK.   We are
grateful to the organizers,  Gui-Qiang Chen, Henrik Shahgholian and Juan Luis V\'{a}zquez,   for both the invitation to participate in the program and to contribute to  this special volume.

\end{document}